\documentclass[reqno,b5paper]{amsart}

\usepackage[T2A]{fontenc}
\usepackage[english]{babel}
\usepackage{graphicx}
\usepackage{amsmath}
\usepackage{amsfonts}
\usepackage{amssymb}
\usepackage{amsthm}

\usepackage{enumerate}
\usepackage[mathscr]{eucal}

\theoremstyle{plain}
\newtheorem{theorem}{Theorem}
\newtheorem{cor}[theorem]{Corrolary}

\newtheorem{lemma}[theorem]{Lemma}

\theoremstyle{remark}

\theoremstyle{definition}
\newtheorem{definition}{Definition}
\DeclareMathOperator{\pr}{pr}
\DeclareMathOperator{\rp}{rp\,}
\DeclareMathOperator{\un}{un}

\begin{document}

\title[Measures on projections in a $W^*$-algebra of type $I_2$]{Measures on projections
in a $W^*$-algebra of type $I_2$}

\author[A.N. Sherstnev]{A.N. Sherstnev*}

\newcommand{\acr}{\newline\indent}

\address{\llap{*} Department of Mathematical Analysis
\acr Kazan Federal University
\acr Kremlyovskaya St., 18 \acr Kazan 420008 \acr RUSSIAN
FEDERATION}

\email{Anatolij.Sherstnev@ksu.ru}



\begin{abstract}
It is shown that for every measure $m$ on
projections in a  $W^*$-algebra of type  $I_2$, there exists a
Hilbert-valued orthogonal vector measure  $\mu$ such that
$\|\mu(p)\|^2= m(p)$ for every projection  $p$. With regard to J.
Hamhalter's result (Proc. Amer. Math. Soc., 110 (1990), 803--806) it
means that the assertion is valid for an arbitrary
$W^*$-algebra.
\end{abstract}

\maketitle

\smallskip
It is well known that  the problem of the extension of a measure on projections to a linear
functional was positively solved for $W^*$-algebras without type $I_2$ direct summand. (A lucid
exposition of Gleason-Christensen-Yeadon'results see in \cite{Mae}.)
In view of  this, it became a good tradition to exclude the $W^*$-algebras with direct summand of type $I_2$ when
measures on projections are investigated.
In this respect, there is an interesting paper by J. Hamhalter \cite{Ham} which describes the connection
between measures
on projections in conventional sense and  $H$-valued ($H$ is a complex Hilbert space)
  orthogonal vector measures.
Specifically, it has been proved in \cite{Ham} (though expressed in a slightly different form) that if  $m$ is a measure
on projections in a  $W^*$-algebra $\mathcal{A}$ without type $I_2$ direct summand, then there exists
a $H$-valued   orthogonal vector measure  $\mu$ on projections in $\mathcal{A}$ such that
 $\|\mu(p)\|^2= m(p)$ for every $p\in\mathcal{A}$. The mentioned proof (in a few lines) of this assertion
 is based on Gleason-Christensen-Yeadon's result.

\smallskip
In this paper we give a construction allowing to obtain a proof of this assertion
for  $W^*$-algebras of type $I_2$ and therefore  for  arbitrary  $W^*$-algebras.
The author is greatly indebted to Lugovaya G.D. for useful discussions.

\bigskip
\centerline{\textbf{Preliminaries}}

\bigskip
 Let $\mathcal{A}$ be a $W^*$-algebra, and $\mathcal{A}^{\pr},\ \mathcal{A}^{\un},\ \mathcal{A}^+$
denote the sets of orthogonal
projections, unitaries, positive elements in $\mathcal{A}$, respectively.
We will denote by $\rp(x)$ the range projection of $x\in \mathcal{A}^+$. It is the least projection
of all projections
$p\in\mathcal{A}^{\pr}$ such that $px=x$. It should be noted that $\rp(x)=\rp(x^{1/2})$.
The basic notions those we talk about in this paper  are described by the following definitions.
(Monograph \cite{She} gives further details of  problems related to  measures on projections
in von Neumann algebras.)

\begin{definition}
 Let  $\mathcal{A}$  be a $W^*$-algebra. A mapping  $m\,:\, \mathcal{A}^{\pr}\to
\mathbb{R}_+$ is called \textit{a measure on projections} if the following condition is  satisfied:
$$
p=\sum\limits_{i\in I}^{}p_i\ \ (p,p_i\in \mathcal{A}^{\pr},\;p_ip_j=0\
(i\ne j))\Rightarrow m(p)=\sum\limits_{i\in I}^{}m(p_i).$$
Here, the series are understood as limits of the nets of finite sums (in $w^*$-topology for projections).
\end{definition}

\begin{definition}
Let $\mathcal{A}$  be a $W^*$-algebra, $H$ be a complex Hilbert space.
A mapping $\mu:\mathcal{A}^{\pr}\to H$ is called \textit{an orthogonal vector measure} if for any set
$(p_j)_{j\in J}\subset \mathcal{A}^{\pr}$ of mutually orthogonal projections the following two conditions
are satisfied:
\begin{itemize}
\item[(i)] the set $(\mu(p_j))_{j\in J}$ is orthogonal in  $H$,
\item[(ii)] $\mu(\sum\limits_{j\in J}^{}p_j)=\sum\limits_{j\in J}^{}\mu(p_j)$,
\end{itemize}
where the series on the right hand side  are understood as the limit
of the net of finite partial sums (in the norm topology on  $H$).

Let $X\subset \mathcal{A}^{\pr}$ has the property
\begin{itemize}
\item[(iii)] $p,q\in X,\ pq=0\Rightarrow p+q\in X$.
\end{itemize}
We call  $\mu: X\to H$  \textit{a finitely additive orthogonal
vector measure on $X$} if the following condition is satisfied
$$
p,q\in X,\ pq=0\ \Rightarrow\   \langle \mu(p),\mu(q)\rangle=0,\ \mu(p+q)=\mu(p)+\mu(q).$$
\end{definition}

\medskip
 We are interested here in  $W^*$-algebras of type $I_2$.
It is  known that the every $W^*$-algebra $\mathcal{N}$ of type $I_2$ can be expressed
in the form $\mathcal{N} =\mathcal{M}\otimes M_2$ where $\mathcal{M}$ is a commutative $W^*$-algebra  and $M_2$
is the algebra of all $2\times 2$ matrices over $\mathbb{C}$.

We turn our attention to the structure of projections in algebra $\mathcal{N}$. We will consider
projections in $\mathcal{N}^{\pr}$ defined as follows:
$$
\pi_1\oplus\pi_2\equiv\left(\begin{array}{cc} \pi_1 & 0\\
0 & \pi_2 \end{array}\right),\qquad \pi_1,\pi_2\in\mathcal{M}^{\pr},$$
$$
p(x,v,\pi)\equiv \left(\begin{array}{cc} x & v(x(\pi-x))^{1/2}\\
v^*(x(\pi-x))^{1/2} & \pi-x \end{array}\right),$$
where $\pi\in\mathcal{M}^{\pr},\ v\in\mathcal{M}^{\un},\ 0\leq x\leq \pi,\
\rp(x(\pi-x))=\pi$. In particular, $p(0,v,0)=0$.

\medskip
The following two lemmas are fairly straightforward from equalities  $p=p^2=p^*$ for
a projection  $p$.

\begin{lemma}
Every projection  $p\in\mathcal{N}^{\pr}$ can be expressed in the form:
$$
p= \pi_1\oplus\pi_2 + p(x,v,\pi),$$
where $\pi_i\leq \textbf{1}-\pi,\ i=1,2$\footnote{Here and subsequently $\textbf{1}$ denote the
identity element in $\mathcal{M}$.}.
\end{lemma}

\medskip
We will denote
$$
\pi\setminus\rho\equiv \pi-\pi\rho,\quad
\pi\Delta\rho\equiv(\pi\setminus\rho)+(\rho\setminus\pi), \qquad \pi,\rho\in\mathcal{M}^{\pr}.$$
Let us observe some useful properties of the mentioned representation for projections.

\begin{lemma}
$p(x,v,\pi)p(y,w,\rho)=0$ if and only if
$$
y\pi\rho=({\bf 1}-x)\pi\rho,\quad w\pi\rho=-v\pi\rho.$$
In addition,
$$
p(x,v,\pi)+p(y,w,\rho)=\pi\rho\oplus \pi\rho + p(z,u,\pi\Delta\rho)$$
where  $z=x(\pi\setminus\rho)+y(\rho\setminus\pi)$ and  $u\in\mathcal{M}^{\un}$ satisfies equations:
$u(\pi\setminus\rho)=v(\pi\setminus\rho),\ u(\rho\setminus\pi)=w(\rho\setminus\pi)$.

 Specifically,  $p(x,v,\pi)p(y,w,\pi)=0$ if and only if $y\pi=(\textbf{1}-x)\pi,\quad w\pi=-v\pi$.
In addition,
$$
p(x,v,\pi)+p(\textbf{1}-x,-v,\pi)=\pi\oplus \pi.$$
\end{lemma}

\begin{lemma}
Let $\mathcal{A}$ be a $W^*$-algebra, $m:\mathcal{A}^{\pr} \to\mathbb{R}_+$ be
a measure on projections and $\mu:\mathcal{A}^{\pr} \to H$ be a
finitely additive orthogonal vector measure with
$$
\|\mu(p)\|^2= m(p),\qquad p\in \mathcal{A}^{\pr}.$$
Then  $\mu$ is the orthogonal vector measure.
\end{lemma}

\begin{proof}
It should be enough to verify the property  (ii) in Definition 2. Let
$p=\sum\limits_{j\in J}^{} p_j= \text{w}{}^*\text{-}\lim\limits_{\sigma}^{}\sum\limits_{j\in\sigma}^{}p_j$
(the limit of the net of finite partial sums). Since (ii) is fulfilled
for finite sums, we have
\begin{align}
\|\mu(p)-\sum\limits_{j\in\sigma}^{}\mu(p_j)\|^2&=\|\mu(p-\sum\limits_{j\in\sigma}^{}p_j)\|^2=m(p-
\sum\limits_{j\in\sigma}^{}p_j)\nonumber\\
&=m(p)-\sum\limits_{j\in\sigma}^{}m(p_j).\nonumber\end{align}
As  $m$ is completely  additive, it follows $\lim\limits_{\sigma}^{}[m(p)-\sum\limits_{j\in\sigma}^{}m(p_j)]=0$.
\end{proof}

We need also the following elementary lemma.

\begin{lemma}
A system of equations
\begin{displaymath}
\left\{\begin{array}{ll}
\lambda_1+\mu_1= &\lambda_0,\\
\lambda_2+\mu_2= &\mu_0,\\
\lambda_1^2+\lambda_2^2= &\lambda^2,\\
\mu_1^2+\mu_2^2= &\mu^2
\end{array}\right.\end{displaymath}
with respect to unknowns $\lambda_i,\mu_i\ (i=1,2)$ where
$$
\lambda_0^2+\mu_0^2=\lambda^2+\mu^2,$$
is solvable in $\mathbb{R}$. In this case
$$
\lambda_1\mu_1+\lambda_2\mu_2=0.$$
\end{lemma}

\bigskip
\centerline{\textbf{A  construction of the orthogonal vector measure}}

\bigskip
 Now  we examine some maximal commutative $W^*$-subalgebras in  $\mathcal{N}$ that will
useful for us. One such subalgebra is  $\mathcal{M}\oplus\mathcal{M}$, the direct sum of two copies of
$\mathcal{M}$,
$$
\mathcal{M}\oplus\mathcal{M}=\left\{\left(\begin{array}{cc} x & 0\\
0 & y \end{array}\right):\ x,y\in\mathcal{M}\right\}.$$

Next,  every pair  $(x,v)$ where
$$
x\in\mathcal{M}^+,\ x\leq\textbf{1},\ \rp(x(\textbf{1}-x))=
\textbf{1},\ v\in\mathcal{M}^{\un},\eqno(1)$$
can be associated with a maximal commutative $W^*$-subalgebra $\mathcal{N}_{x,v}$ in
$\mathcal{N}$ described by  the set of its projections
$$
\mathcal{N}_{x,v}^{\pr}=\{p(x\pi_1,v,\pi_1)+p((\textbf{1}-x)\pi_2,-v,\pi_2):\
\pi_i\in\mathcal{M}^{\pr},\ i=1,2\}.$$
It is easily seen that  $\mathcal{N}_{x,v}$ is maximal. Note that $\mathcal{N}_{x,v}=\mathcal{N}_{\textbf{1}-x,-v}$.

\bigskip Let us to index the set of all such pairs, and associate to each  $\gamma=(x,v)$
the set  $\mathcal{N}^{\pr}_{\gamma}\equiv \mathcal{N}^{\pr}_{x,v}$ and  associate to  $0$ the set
$\mathcal{N}^{\pr}_{0}\equiv\{\pi_1\oplus\pi_2: \pi_i\in \mathcal{M}^{\pr},i=1,2\}$.
Then we totally order the set  $\Gamma$ of all indices $\gamma$, taking $0=\min \Gamma$.

\bigskip
It is known (\cite[Proposition 1.18.1]{Sak}) that a commutative  $W^*$-algebra $\mathcal{M}$ may be
realized as $C^*$-algebra  $L^{\infty}(\Omega,\nu)$ of all essentially bounded locally $\nu$-measurable
functions on a localizable measure space ($\Omega, \nu)$ (i. e.  $\Omega$ is direct sum of finite measure
spaces, see \cite{Seg}). In this case, the Banach space $L^1(\Omega,\nu)$ is the predual of $L^{\infty}(\Omega,\nu)$:
$L^1(\Omega,\nu)^*=L^{\infty}(\Omega,\nu)$. Now  we shall identify  $\mathcal{M}$ with $L^{\infty}(\Omega,\nu)$.
In this case the characteristic functions
$$
\pi(\omega)\equiv\chi_{\pi}(\omega)=\left\{\begin{array}{ll}
1, &\text{if } \omega\in\pi,\\
0, & \text{if }\omega\not\in\pi,\end{array}\right.\qquad \pi\subset\Omega, $$
correspond to projections $\pi\in\mathcal{M}^{\pr}$.
(The reader will note to his regret that we use the same letter $\pi$ to designate three objects:
a projection in  $\mathcal{M}^{\pr}$, a $\nu$-measurable set in $\Omega$,  the characteristic function
of this set.) By virtue of classical integration theory, for every measure
$\sigma: L^{\infty}(\Omega,\nu)^{\pr}\to\mathbb{R}_+$ is determined uniquely a function
$h\in L^1(\Omega,\nu),\ h\geq0$, such that
$$
\sigma(\pi)=\int\pi(\omega)h(\omega)\nu(d\omega)=\int\limits_{\pi}h(\omega)\nu(d\omega).$$
In this approach, the  $W^*$-algebra $\mathcal{N}$ is realized as von Neumann algebra of
$2\times 2$-matrices  $(x_{ij}),\ x_{ij}\in L^{\infty}(\Omega,\nu)$ acting on
the orthogonal sum of two copies of Hilbert space  $L^2(\Omega,\nu)$:
$$
H=L^2(\Omega,\nu)\dotplus L^2(\Omega,\nu)=\left\{\left(\begin{matrix}f\\
g\end{matrix}\right):\ f,g\in L^2(\Omega,\nu)\right\}.$$

Next, let  $m:\mathcal{N}^{\pr}\to\mathbb{R}_+$ be a given measure on projections on $W^*$-algebra
$L^{\infty}(\Omega,\nu)\otimes M_2$. Let $0\leq h_0,k_0\in L^2(\Omega,\nu)$ such that
$$
m(\pi_1\oplus\pi_2)=\int\limits_{\pi_1}h_0^2(\omega)\nu(d\omega)+
\int\limits_{\pi_2}k_0^2(\omega)\nu(d\omega), \quad \pi_i\in\mathcal{M}^{\pr}.$$

\bigskip Similarly, there are $0\leq h_{\gamma},k_{\gamma}\in L^2(\Omega,\nu)$ such that
$$
m(p(x\pi,v,\pi))=\int\limits_{\pi}h_{\gamma}^2\,d\nu,\quad m(p((\textbf{1}-x)\pi,-v,\pi))=
\int\limits_{\pi}k_{\gamma}^2\,d\nu,\quad \pi\in\mathcal{M}^{\pr}.\eqno(2)$$

\bigskip
In addition,  Lemma 2 and the Radon-Nykodim theorem give
$$
h_{\gamma}^2(\omega)+k_{\gamma}^2(\omega)=h_0^2(\omega)+k_0^2(\omega)\quad\text{a.~e.}$$

\bigskip
We will now state the main result of this paper.

\begin{theorem}
Let $m:\mathcal{N}^{\pr}\to\mathbb{R}_+$ be a measure on projections in $W^*$-algebra $\mathcal{N}$ of
type $I_2$. Then there exist a Hilbert space  $H$ and an orthogonal vector measure
$\mu:\mathcal{N}^{\pr}\to H$ with property
$$
\|\mu(p)\|^2= m(p),\qquad p\in \mathcal{N}^{\pr}.$$
\end{theorem}

\begin{proof}
Define an orthogonal vector measure $\mu$ on the set $[0]\equiv \mathcal{N}_0^{\pr}$ via
$$
\mu(\pi_1\oplus\pi_2)\equiv\left(\begin{matrix}\pi_1h_0\\
\pi_2k_0\end{matrix}\right),\quad \pi_1,\pi_2\in\mathcal{M}^{\pr}.$$
We next extend $\mu$ to an orthogonal vector measure on the set $[0,1]$ of all projections in
the form
$$
p= \pi_1\oplus\pi_2 + p(x\pi_3,v,\pi_3)+ p((\textbf{1}-x)\pi_4,-v,\pi_4),\quad\pi_i\in\mathcal{M}^{\pr},\eqno(3)$$
where $(x,v)$ is a pair in (1) corresponding  to index $1\equiv\min (\Gamma\setminus\{0\})$.
According to Lemma 2, it is possible to assume that $\pi_3\pi_4=0$. Thus,
$$
\pi_1\pi_3=\pi_1\pi_4=\pi_2\pi_3=\pi_2\pi_4=\pi_3\pi_4=0.$$
One can easily see that the set $[0,1]$ satisfies (iii) in Definition 2. In view of Lemma 3 there are
real functions $0\leq h_{1i},k_{1i}\in L^2(\Omega,\nu), i=1,2$ such that equalities
\begin{align}
h_{11}(\omega)+k_{11}(\omega)&=h_0(\omega),\tag*{(4)}\\
h_{12}(\omega)+k_{12}(\omega)&=k_0(\omega),\tag*{(5)}\\
h_{11}^2(\omega)+h_{12}^2(\omega)&=h_1^2(\omega),\tag*{(6)}\\
k_{11}^2(\omega)+k_{12}^2(\omega)&=k_1^2(\omega),\tag*{(7)}
\end{align}
$$
h_{11}(\omega)k_{11}(\omega)+h_{12}(\omega)k_{12}(\omega)=0.\eqno(8)$$
hold a.~e. (Here the functions $h_1,k_1$ in (2) correspond to index $\gamma=1$.)

We now extend the function $\mu$ to projections in the form (3) putting
$$
\mu(p)\equiv\left(\begin{matrix}\pi_1h_0+\pi_3h_{11}+\pi_4k_{11}\\
\pi_2k_0+\pi_3h_{12}+\pi_4k_{12}\end{matrix}\right),\eqno(27)$$
where  $h_{1i},k_{1i}$ are solutions of the system (4) --- (7).
Direct computations with application (4) -- (8) show that $\mu$ is
the finitely additive orthogonal vector measure on $[0,1]$.

Suppose (inductive hypothesis) that $\mu$ is extended to a finitely additive orthogonal vector
measure on $[0,\gamma)$,
$$
[0,\gamma)\equiv \{p_1+...+p_s: p_j\in\mathcal{N}^{\pr}_{\gamma_j}, p_jp_k=0\ (j\ne k),\ \gamma_j<\gamma\},$$
and $\gamma=(y,w)$. Let $0\leq h_{\gamma}, k_{\gamma}\in L^2(\Omega,\nu)$ such that
$$
m(p_{\gamma})=\int\limits_{\rho_1}h_{\gamma}^2\,d\nu+\int\limits_{\rho_2}k_{\gamma}^2\,d\nu,\eqno(9)$$
where
$$
p_{\gamma}= p(y\rho_1,w,\rho_1)+p((\textbf{1}-y)\rho_2,-w,\rho_2),\quad \rho_1\rho_2\in\mathcal{M}^{\pr}. \eqno(10)$$
Let
$$
\mathcal{P}=\{\pi\in\mathcal{M}^{\pr}: \exists (x,v)<(y,w)\ (x\pi=y\pi,\ v\pi=w\pi)\},$$
and  $(\pi_j)_{j\in J}\subset\mathcal{P}$ be a maximal set of pairwise orthogonal projections in $\mathcal{P}$
(it exists by Zorn's theorem). Define  $\pi_0\equiv\sum\limits_{j}^{}\pi_j(=\sup\mathcal{P})$.

With the above notations we have
\begin{align}
p(y,w,\textbf{1})&=p(y(\textbf{1}-\pi_0),w,\textbf{1}-\pi_0)+\sum\limits_{j}^{}p(y\pi_j,w,\pi_j)\nonumber\\
&=p(y(\textbf{1}-\pi_0),w,\textbf{1}-\pi_0)+\sum\limits_{j}^{}p(x_j\pi_j,v_j,\pi_j),\nonumber\\
p(\textbf{1}-y,-w,\textbf{1})&=p((\textbf{1}-y)(\textbf{1}-\pi_0),-w,\textbf{1}-\pi_0)+
\sum\limits_{j}^{}p((\textbf{1}-x_j)\pi_j,-v_j,\pi_j),
\nonumber\end{align}
where  $(x_j,v_j)<(y,w)$.

By inductive hypothesis there defined the functions $h_{j1},k_{j1},h_{j2},k_{j2}\in L^2(\Omega,\nu)$
satisfying equalities
\begin{align}
h_{j1}(\omega)+k_{j1}(\omega)&=h_0(\omega),\nonumber\\
h_{j2}(\omega)+k_{j2}(\omega)&=k_0(\omega),\nonumber\\
h_{j1}^2(\omega)+h_{j2}^2(\omega)&=h_j^2(\omega),\nonumber\\
k_{j1}^2(\omega)+k_{j2}^2(\omega)&=k_j^2(\omega),\nonumber
\end{align}
$$
h_{j1}(\omega)k_{j1}(\omega)+h_{j2}(\omega)k_{j2}(\omega)=0.$$
where the density functions  $h_j,k_j$ correspond to pairs $(x_j,v_j)$ according to (2).
We also find the functions $\widetilde h_{\gamma 1},\widetilde k_{\gamma 1},\widetilde h_{\gamma 2},
\widetilde k_{\gamma 2}\in L^2(\Omega,\nu)$ that are solutions of equations
 \begin{align}
\widetilde h_{\gamma 1}(\omega)+\widetilde k_{\gamma 1}(\omega)&=h_0(\omega),\nonumber\\
\widetilde h_{\gamma 2}(\omega)+\widetilde k_{\gamma 2}(\omega)&=k_0(\omega),\nonumber\\
\widetilde h_{\gamma 1}^2(\omega)+\widetilde h_{\gamma 2}^2(\omega)&=h_{\gamma}^2(\omega),\nonumber\\
\widetilde k_{\gamma 1}^2(\omega)+\widetilde k_{\gamma 1}^2(\omega)&=k_{\gamma}^2(\omega),\nonumber
\end{align}
where $h_{\gamma}, h_{\gamma}$ are defined by (9). Therefore, there are defined the functions
 \begin{align}
h_{\gamma 1}(\omega)&\equiv (\textbf{1}-\pi_0)\widetilde h_{\gamma 1}(\omega)+
\sum\limits_{j}^{}\pi_j(\omega)h_{j1}(\omega),\nonumber\\
 h_{\gamma 2}(\omega)&\equiv (\textbf{1}-\pi_0)\widetilde h_{\gamma 2}(\omega)+
 \sum\limits_{j}^{}\pi_j(\omega)h_{j2}(\omega),\nonumber\\
k_{\gamma 1}(\omega)&\equiv (\textbf{1}-\pi_0)^2(\omega)\widetilde k_{\gamma 1}(\omega)+
\sum\limits_{j}^{}\pi_j(\omega)k_{j1}(\omega),\nonumber\\
k_{\gamma 2}(\omega)&\equiv (\textbf{1}-\pi_0)\widetilde k_{\gamma 2}(\omega)+
\sum\limits_{j}^{}\pi_j(\omega)k_{j2}(\omega).\nonumber
\end{align}
In this case
$$
h_{\gamma 1}^2(\omega)+h_{\gamma 2}^2(\omega)=h_{\gamma}^2(\omega),\quad k_{\gamma 1}^2(\omega)+
k_{\gamma 2}^2(\omega)=k_{\gamma}^2(\omega)\quad \text{a.~e.,}$$
$$
h_{\gamma 1}(\omega)k_{\gamma 1}(\omega)+ h_{\gamma 2}(\omega)k_{\gamma 2}(\omega)=0\quad \text{a.~e.}$$
Now we put
$$
\mu(p+p_{\gamma })\equiv \mu(p)+\left(\begin{matrix}\rho_1h_{\gamma  1}+\rho_2k_{\gamma  1}\\
\rho_1h_{\gamma  2}+\rho_2k_{\gamma  2}\end{matrix}\right).$$
where $p\in [0,\gamma)$ and $p_{\gamma}$ is defined by (10). Again, direct computations show that
$\mu$ is the finitely additive orthogonal vector
measure on $[0,\gamma]$.

In view of Lemma 1 it follows that $\mu$ turned out extended to $\mathcal{N}^{\pr}$. Applying
Lemma 3 we complete the proof.
\end{proof}

\begin{cor}
Let  $m:\mathcal{A}^{\pr}\to\mathbb{R}_+$ be a measure on projections in an
arbitrary  $W^*$-algebra $\mathcal{A}$. Then there exist complex Hilbert space
 $H$ and an orthogonal vector measure $\mu:\mathcal{A}^{\pr}\to H$ such that
$$
\|\mu(p)\|^2=m(p),\qquad p\in\mathcal{A}^{\pr}.$$
\end{cor}

\begin{proof}
Because an  orthogonal vector measure is uniquely defined by its restrictions
to direct summands of a $W^*$-algebra, the statement follows by Theorem 5 and the proof of Theorem in
\cite{Ham}.
\end{proof}

\end{document}